\documentclass[runningheads,a4paper,envcountsame]{llncs}

\usepackage{amsfonts}
\usepackage{amssymb}
\usepackage{mathtools}
\usepackage[all,cmtip]{xy}
\usepackage[bookmarks=false]{hyperref}
\usepackage{xcolor}
\usepackage{todonotes}
\usepackage{url}
\newtheorem{defini}{Definition}
\newtheorem{thm}{Theorem}

\renewcommand{\phi}{\varphi}

\title{A Note on Power-OTMs}
\author{Merlin Carl}
\date{November 2024}
\institute{Institut f\"ur Mathematik, Europa-Universit\"at Flensburg}

\begin{document}

\maketitle

\begin{abstract}
We consider the computational strength of Power-OTMs, i.e., ordinal Turing machines equipped with a power set operator, and study a notion of realizability based on these machines. When parameters are allowed, these machines are, modulo access to a global well-ordering, equivalent to the Set Register Machines defined by Robert Passmann in \cite{Passmann}, and while most of the results on the realizability of Power-OTMs are analogous to results obtained by Passmann, the settings lead to different results concerning the axiom of choice. As we will see, the computational strength of power-OTMs can, depending on the set-theoretical background, also differ from that of Set Register Machines.
\end{abstract}

%ggf. interessant: https://mathoverflow.net/questions/180727/is-it-consistent-with-zfc-or-zf-that-every-definable-family-of-sets-has-at-lea/180734#180734

%FRAGE: Gilt Sacks für parameterfreie Power-OTMs?

\section{Introduction}

Ordinal Turing Machines (OTMs) are a model of transfinite computability introduced by Peter Koepke in \cite{Koepke:OTMs}. A notion of realizability based on OTMs was considered in \cite{CGP}; the guiding idea here is to consider OTM-computable functions as the underlying notion of effectivity on which the realizability notion is then based. A basic result on OTMs is that (provably in ZFC), OTMs cannot compute power sets. However, it has been argued that the formation of power sets should be regarded as a basic constructive operation in set theory, and realizability notions based on this idea have been defined and studied, see, e.g., Rathjen \cite{Rathjen:Power}, p. 3 or Moss \cite{Moss}, p. 248. This motivates the study of OTMs equipped with an extra command for forming power sets, which we call \textit{power-OTMs}. After finishing this note, we became aware that a very similar investigation was conducted earlier by Robert Passmann in \cite{Passmann}, whose ``Set Register Machines'' are essentially a register variant of power-OTMs, designed to avoid coding issues (the equivalence of the two models when the power-set operator is dropped is proved in \cite{Passmann}, Theorem 3.8). Below, we will indicate the corresponding results about SRMs. It still seems to be worth noting that these results can be integrated in the more familiar OTM-setting. Passmann allows his SRMs the use of arbitrary set parameters; in contrast, we consider only the cases where (i) a single ordinal parameter and (ii) no parameters are allowed. The results concerning the latter seem to be essentially new. Moreover, unlike Passmann, we do not assume the existence global well-ordering, which changes the realizability of the axiom of choice.

\section{Basic Definitions and Results}

We assume that the reader is familiar with Koepke's Ordinal Turing Machines (OTMs), see \cite{Koepke:OTMs} or \cite{CarlBuch}, section 2.5.6. For convenience, we will consider OTMs with four tapes: An input tape, a scratch tape, an output tape and an extra tape for carrying out the operations to which we are reducing. 

\begin{defini}{\label{F-OTM}}
Let $F:V\rightarrow V$ be a class function. An $F$-encoding is a map $\hat{F}:\mathfrak{P}(\text{On})\rightarrow\mathfrak{P}(\text{On})$ which, for every $c_{x}\subseteq\text{On}$ that codes a set $x$, maps $c_{x}$ to a code $c_{F(x)}\subseteq\text{On}$ for $F(x)$. An $F$-OTM-program $P$ is an OTM-program using four tapes that has as an additional command an $F$-command. We call the four tapes the input tape, the scratch tape, the output tape and the $F$-tape. 

If $\hat{F}$ is an $F$-encoding, then the computation of $P$ (on a given input, $x$ if there is one) is obtained by running the program as usual (according to the definition of OTM-computations in \cite{Koepke:OTMs}), and, whenever the $F$-command is carried out, replacing the current content $c$ of the $F$-tape by $\hat{F}(c)$. We denote this computation by $P^{\hat{F}}(x)$. 

If, for all inputs, the set encoded by the output of $P$ is the same for all encodings of $F$ and all encodings of the input $x$, we call $F$ \textit{coding-stable}.\footnote{Similar notions are used in \cite{CGP} (there called \textit{uniformity}) or \cite{Ca2016}.} In this case, we will denote the set encoded by the output of $P^{\hat{F}}(c)$ for any encoding $\hat{F}$ of $F$ and any code $c_{x}$ for a set $x$ by $P^{F}(x)$. 

A set $X$ is called $F$-OTM-computable if there are a coding-stable $F$-OTM-program $P$ (and perhaps an ordinal parameter $\rho\in\text{On}$) that halts with a code for $X$ on its output tape. $X$ is called $F$-OTM-\textit{recognizable}\footnote{The notion of recognizability was first defined by Hamkins and Lewis for Infinite Time Turing Machines in \cite{Hamkins-Lewis}, see Theorem 4.9.} if and only if there is a coding-stable OTM-program $P$ (and, perhaps, a parameter $\rho\in\text{On}$) such that, for all $s\subseteq\text{On}$: $$P^{F}(s,\rho)=\begin{cases}1\text{, if }s\text{ codes }x\\0\text{, otherwise}\end{cases}$$
IF we can take $\rho=0$, we speak of parameter-free $F$-OTM-computability and -recognizability, respectively. 

The definition of $F$-realizability is obtained from the definition of uniform OTM-realizability in \cite{CGP} by replacing OTMs with set parameters by coding-stable $F$-OTMs everyhwere. 

Similarly, the definition of recognizable $F$-realizability, or $F$-$r$-realizability, is obtained from the definition of $r$-realizability in \cite{Ca2024} by replacing OTMs with coding-stable $F$-OTMs everywhere.

If $F$ is the power set operator $x\mapsto\mathfrak{P}(x)$, we speak of \textit{power-OTMs} and \textit{power-realizability}. 

We fix a natural enumeration $(P_{i}:i\in\omega)$ of power-OTM-programs. 
\end{defini}

\begin{remark}
Power-OTM-computability is conceptually closely connected to Moss' concept of ``power-set recursion'', defined in \cite{Moss}. We will see below that power-set recursion is strictly weaker than parameter-free power-OTM-computability. Moreover, Moss does not consider realizability in his work, which is the main motivation of the present paper.    
\end{remark}

An important feature of OTMs is their ability to enumerate Gödel's constructible universe $L$. Similarly, Power-OTMs can enumerate the set-theoretical universe $V$:

\begin{lemma}{\label{enumerating the universe}} [Cf. Passmann, \cite{Passmann}]
\begin{enumerate}
\item There is a (parameter-free) non-halting power-OTM-program $P_{\text{level-enum}}$ such that, for every ordinal $\alpha$, there is an ordinal $\tau$ such that, at time $\tau$, $P_{\text{level-enum}}$ has a code for $V_{\alpha}$ on its output tape. Moreover, there is a non-halting power-OTM $P_{\text{set-enum}}$ such that, for every set $x$, there is an ordinal $\tau$ such that, at time $\tau$, $P_{\text{set-enum}}$ has a code for $x$ on its output tape.\footnote{This is essentially Passmann, \cite{Passmann}, Proposition 3.12.}
\item There is a (parameter-free) power-OTM-program $P_{\text{level}}$ such that, for every code $c_{x}$ for a set $x$, $P_{\text{level}}(x)$ halts with a code for the $\in$-minimal $V_{\alpha}\ni x$ on its output tape.\footnote{See Passmann, \cite{Passmann}, Proposition 3.10.}
\item There is a (parameter-free) power-OTM-program $P_{\text{card}}$ which, given a code $c_{x}$ for a set $x$, computes $\text{card}(x)$.
\end{enumerate}
\end{lemma}
\begin{proof}
We give a brief sketch. (The arguments work as those of Passmann quoted above.) $P_{\text{level-enum}}$ is obtained by starting with the empty set, iteratively applying the power set operator, storing the intermediate results on the working tape and computing unions at limit times. $P_{\text{set-enum}}$ is then obtained from $L_{\text{level-enum}}$ by additionally running through each newly written level and writing codes for all of its elements to the output tape. If a set $x$ is given, one can search each level occuring in the run of $P_{\text{level-enum}}$ to check whether it contains $x$ and halt once the answer is positive, which yields $P_{\text{level}}$. Finally, for $P_{\text{card}}$, first use $P_{\text{level-enum}}$ to find some $V_{\alpha}\ni x$ that contains a bijection between $x$ and some ordinal for some given set $x$, and then search through $V_{\alpha}$ for the minimal ordinal for which there is such a bijection.\footnote{This is essentially the argument for the reverse direction of Passmann, \cite{Passmann}, Proposition 3.10.} 
\end{proof}

The ability of power-OTMs to enumerate the whole set-theoretical universe has the consequence that computability and recognizability coincide:

\begin{thm}{\label{no lost melodies}}
A set is power-OTM-computable if and only if it is power-OTM-recognizable. In particular, power-realizability coincides with power-r-realizability.%\todo{power-r-realizability noch definieren!}
\end{thm}
\begin{proof}
    As usual, computable sets can be recognized by computing (a code for) them and then comparing that to the oracle. 

    In the other direction, if $Q$ is a power-OTM-program (possibly with ordinal parameters) that recognizes a set $x$, we can compute $x$ by running $P_{\text{level-enum}}$ and checking all elements of all occuring levels with $Q$. Eventually, a code for $x$ will be produced and recognized, in which case the program writes that code to the output tape and halts.
\end{proof}

This shows that the whole set-theoretical universe is ``recursively enumerable'' in the sense of Power-OTMs. Concerning the computational strength with ordinal parameters, we obtain the following:\footnote{This should be contrasted with Passmann's setting, where, since arbitrary set parameters are allowed, every set is trivially SRM-computable; and even if only ordinal parameters were allowed, then, since the machine is given access to a global well-ordering, every set would still be SRM$^{+}$-computable.}

\begin{thm}{\label{comp strength power-otms}}
A set $X$ is power-computable if and only if it is $\Sigma_{2}$-definable with ordinal parameters (i.e., an element of ``$\Sigma_{2}$-OD'').
\end{thm}
\begin{proof}
Suppose that $X$ is power-OTM-computable, and let $P$ be a power-program that computes $X$ in the parameter $\rho\in\text{On}$. For any particular encoding $\hat{\mathfrak{P}}$ of the power set class function, the halting computation $P^{\hat{\mathfrak{P}}}(\rho)$ will only use set many instances of $\mathfrak{P}$. Thus, $X$ is definable by the formula stating that there is a (set-sized) encoding of a partial (set-sized) function $f\subseteq\mathfrak{P}$ such that $P^{f}(\rho)$ halts without requiring instances of $\mathfrak{P}$ not coded in $f$ and with an output that codes $X$. This is expressible as a $\Sigma_{2}$-formula. %details später nachliefern: für $f$ reicht es, die potenzmenge einer hinreichend großen $V$-Stufe zu kennen, die Charakterisierung davon ist $\Sigma_{2}$...

Now suppose that $X$ is ordinal definable by a $\Sigma_{2}$-formula in the ordinal parameter $\rho$, and let $\psi(\rho,z):\Leftrightarrow\exists{x}\forall{y}\phi(x,y,\rho,z)$ (with $\phi$ a $\Delta_{0}$-formula) be such that $X$ is the unique $z$ for which this formula is true. We now explain how to power-compute $X$. Note that for any $a$ and $b$ for which $\forall{y}\phi(a,y,\rho,b)$ is false, there is such $y$ in $V_{\kappa^{+}}$ with $\kappa=\text{max}\{\text{card}(\text{tc}(a)),\text{card}(\text{tc}(b))\}$ by Lemma \ref{early witnesses} below. Thus, statements of this form are power-OTM-decidable by, for given $a$ and $b$, computing $\text{tc}(a)$ and $\text{tc}(b)$ (which can be done on a plain OTM), then computing their cardinalities (by Lemma \ref{enumerating the universe}(3), taking their maximum $\kappa$, computing its cardinal successor $\kappa^{+}$ (which can be done by first computing the cardinality $\delta$ of $\mathfrak{P}(\kappa)$ and then testing for each ordinal $\alpha$ between $\kappa$ and $\delta$ whether there is a bijection between $\alpha$ and $\kappa$, which, in turn, can be done by searching $\mathfrak{P}(\kappa\times\delta)$), computing $V_{\kappa^{+}}$ (as in  Lemma \ref{enumerating the universe}(2)) and searching $V_{\kappa^{+}}$ for a counterexample.
Now, to compute $X$, we proceed as follows: As described in Lemma \ref{enumerating the universe}, we enumerate $V$. For every ordered pair $(a,b)$ of sets occuring during this enumeration, we test in the way just desribed whether $\forall{y}\phi(a,y,\rho,b)$ is true. Since $\psi(\rho,X)$ holds by assumption, such a pair must eventually be found. Once it is found, we return $b$ (i.e., the second component of the pair). By definition, $b$ will be equal to $X$.
\end{proof}

\begin{remark}
As a consequence, the computational strength of power-OTMs with ordinal parameters is different from that of SRM$^{+}$ that use ordinal parameters in set-theroretical universes in which if not every element ordinal definable, such as a Cohen-generic extension of $L$. 
\end{remark}

%DAS FOLGENDE IST FALSCH! berechenbarkeit erfordert ja unabhängigkeit von der codierung von $F$!
%\begin{corollary}{\label{comp strength power-otms}} [Cf. Passmann, \cite{Passmann}, Proposition 3.12]
%Every set is computable by Power-OTMs with ordinal parameters.
%\end{corollary}

Without ordinal parameters, we need to work a bit more. 

\begin{defini}{\label{sigma P def}}
Let $\mathcal{L}_{\in,\mathcal{P}}$ denote the extension of the language of set theory with the one-place function symbol $\mathcal{P}$ (to be interpreted as the power set operator).\footnote{This language was considered, e.g., by Rathjen, \cite{Rathjen:Power}.} 
Let $\Sigma_{1}^{\mathcal{P}}$ denote the set of $\Sigma_{1}$-formulas in that language.\footnote{Cf. ibid.} For every $\Sigma_{1}^{\mathcal{P}}$-formula $\phi$, define $$\alpha(\phi):=\begin{cases}\text{the smallest ordinal }\alpha\text{ such that }V_{\alpha}\models\phi\text{, if it exists}\\0\text{, otherwise.}\end{cases}$$
Finally, let $\sigma^{\mathcal{P}}:=\text{sup}\{\alpha(\phi):\phi\in\Sigma_{1}^{\mathcal{P}}\}$. We  fix a natural enumeration $(\psi_{i}:i\in\omega)$ of the $\Sigma_{1}^{\mathcal{P}}$-formulas. 
\end{defini}

\begin{lemma}{\label{no max sigma1 p}}
    The set $\Phi:=\{\alpha(\phi):\phi\in\Sigma_{1}^{\mathcal{P}}\}$ has no maximum. 
\end{lemma}

The following lemma is part of the folklore. We include a proof for the sake of completeness.

\begin{lemma}{\label{early witnesses}}
Let $y$ be a set, $\psi$ be $\Delta_{0}$, $t:=\text{tc}(\{y\})$ and $\kappa:=\text{card}(t)$. If $\exists{x}\psi(x,y)$ then there exists some $x$ with $x\in V_{\kappa^{+}}$ so that $\psi(x,y)$.%\todo{kommt dieser trick auch bei moss vor?} -- anscheinend nicht.
\end{lemma}

%Kunen, Lemma 6.2: Für jede unendliche Kardinalzahl $\kappa$ gilt $H_{\kappa}\subseteq V_{\kappa}$

\begin{proof} Without loss of generality, let $\kappa\geq\omega$. By \cite{Kunen}, Lemma 6.2, p. 131, this implies $t\in V_{\kappa^{+}}$. 
%Es ist $t\in V_{\alpha+1}$, denn da $y\subseteq V_{\alpha}$ und $V_{\alpha}$ transitiv ist, folgt $t\subseteq V_{\alpha}$, also $t\in V_{\alpha+1}$.
 Now let $\beta$ be minimal such that $x\in V_{\beta}$; without loss of generality, we assume that $\beta>\kappa$. Let $H$ bet the transitive collapse of the $\Sigma_{1}$-hull of $t\cup\{x\}$ in $V_{\beta}$, and let $\pi:\Sigma_{1}\{t\cup\{x\}\}\rightarrow H$ be the collapse map. Then $\kappa\leq\text{card}(H)\leq\kappa\omega=\kappa$, so $\text{card}(H)=\kappa$. Moreover, $H$ is transitive, so we have $H\in H_{\kappa^{+}}$.\footnote{Here, $H_{\kappa}$ refers, as usual, to the set of sets $x$ whose transitive closure has cardinality strictly less than $\kappa$.} Since $t\subseteq H$, we have $\pi(y)=y$. Now $H\models\psi(\pi(x),y)$, and since $\psi$ is a $\Delta_{0}$-formula, we have $\psi(\pi(x),y)$ in $V$. Because of $\pi(x)\in H \in H_{\kappa^{+}}\subseteq V_{\kappa^{+}}$ (again because of \cite{Kunen}, Lemma 6.2) and because $V_{\kappa^{+}}$ is transitive, we have $\pi(x)\in V_{\kappa^{+}}$.
\end{proof}

\begin{thm}
\begin{enumerate}
\item $\sigma^{\mathcal{P}}$ is the smallest ordinal $\alpha$ such that $V_{\alpha}$ contains all sets that are computable by a parameter-free power-OTM.
\item In particular, there are cofinally in $\sigma^{\mathcal{P}}$ many ordinals $\alpha$ such that $V_{\alpha}$ is power-OTM-computable without parameters.
\item $\sigma^{\mathcal{P}}$ is a limit cardinal with cofinality $\omega$.
\item There are a parameter-freely power-OTM-computable set $X$ and an $\in$-formula $\phi$ such that $\{x\in X:\phi(x,\rho)\}$ is not power-OTM-computable.
\item The set of power-OTM-computable sets is not transitive. 
\item $V_{\sigma^{\mathcal{P}}}$ is admissible, but not $\Sigma_{2}$-admissible. 
\end{enumerate}
\end{thm}
\begin{proof}
\begin{enumerate}
\item Let $P$ be a halting power-OTM-program. 
The proposition that $P$ halts is $\Sigma_{1}^{\mathcal{P}}$ and, by assumption, it holds in $V$; thus, by definition of $\sigma^{\mathcal{P}}$, it also holds in $V_{\sigma^{\mathcal{P}}}$. Thus, the whole computation of $P$ exists in $V_{\sigma^{\mathcal{P}}}$, and since $V_{\sigma^{\mathcal{P}}}$ is transitive, so does its output. Consequently, every power-OTM-computable set is an element of 
 $V_{\sigma^{\mathcal{P}}}$. 

Conversely, let $\alpha<\sigma^{\mathcal{P}}$. By definition of $\sigma_{1}^{\mathcal{P}}$, there exists some $\Sigma_{1}^{\mathcal{P}}$-formula $\phi$ such that $\alpha(\phi)>\alpha$. But then, $V_{\alpha(\phi)}$ is power-OTM-computable without parameters: To see this, use $P_{\text{level-enum}}$ and check, for every generated $V$-stage $V_{\beta}$, whether $V_{\beta}\models\phi$. As soon as this happens to be true (which is guaranteed to eventually be the case), the program halts and outputs a code for $V_{\beta}$.
\item This follows immediately from the proof for (1). 
\item Assume for a contradiction that $\sigma^{\mathcal{P}}$ is not a cardinal. Thus, there exists an ordinal $\alpha<\sigma^{\mathcal{P}}$ with $\text{card}(\alpha)=\text{card}(\sigma^{\mathcal{P}})$. By (2), there exists a (parameter-freely) power-OTM-computable stage $V_{\beta}$ with $\beta\geq\alpha$. Now $\text{card}(V_{\beta})\geq\alpha$. But if $V_{\beta}$ is power-OTM-computable without parameter, then so is $V_{\beta+1}=\mathcal{P}(V_{\beta})$, and thus, by Lemma \ref{enumerating the universe}, so is $\kappa:=\text{card}(V_{\beta+1})$. It follows that $\kappa\in V_{\sigma^{\mathcal{P}}}$, so $\kappa<\sigma^{\mathcal{P}}$, and this implies $\text{card}(\alpha)<\kappa\leq\sigma^{\mathcal{P}}$, a contradiction.

%Nach Lemma \ref{power otm basics} ist $\kappa:=\text{card}(V_{\beta})$ Power-OTM-berechenbar. 
%Folglich ist auch auch $\mathcal{P}(\kappa)$ Power-OTM-berechenbar, und also auch \lambda:=\text{card}($\mathcal{P}(\kappa)\in V_{\sigma^{\mathcal{P}}$. Wegen $\alpha\leq\beta<\kappa$ folgt $\text{card}(\alpha)<\text{\card}(\mathcal{P}(\kappa))$. Da $\
% und eine Surjektion $g:\beta\rightarrow\sigma^{\mathcal{P}}$. 

Thus $\sigma^{P}$ is a cardinal. Concerning its cofinality, the function $f:\omega\rightarrow V_{\sigma^{\mathcal{P}}}$ with $f(k)=\alpha(\psi)$ is $\Sigma_{2}$ over $V_{\sigma^{\mathcal{P}}}$ and we have $\omega\in V_{\omega^{\mathcal{P}}}$. But, by definition of $\sigma^{\mathcal{P}}$,  $f[\omega]$ is is unbounded in $\sigma^{\mathcal{P}}$. Hence $\text{cf}(\sigma^{\mathcal{P}})=\omega$. 
%Ist $V_{\alpha}$ Power-OTM-berechenbar, so offenbar auch $\mathcal{P}(V_{\alpha})$. Also ist $\sigma^{\mathcal{P}}$ eine Limes
\item Since the power set operator is definable by an $\in$-formula, the statement $k\in\omega\wedge \mathcal{P}_{k}\downarrow$ is expressable as an $\in$-formula $\phi(k)$. If one uses separation with this formula from $\omega$, one obtains the halting set $h_{\mathcal{P}}$ for power-OTMs, which is not power-OTM-computable (which can be seen by the usual diagonal argument). 
\item Clearly, $\mathcal{P}(\omega)$ is power-OTM-computable, but, by (4), the same does not hold for $h_{\mathcal{P}}\in\mathcal{P}(\omega)$.
\item Separation, Pot and AC hold in every $V$-stage $V_{\delta}$, provided $\delta$ is a limit cardinal; the same holds for the KP-axioms, with the exception of $\Delta_{0}$-Collection. 

To see $\Delta_{0}$-collection, consider a $\Delta_{0}$-formula $\psi(x,y)$ and a set $X\in V_{\sigma^{\mathcal{P}}}$ with $V_{\sigma^{\mathcal{P}}}\models\forall{x\in X}\exists{y}\psi(x,y)$. Pick $\gamma<\sigma^{\mathcal{P}}$ such that $x\in V_{\gamma}$ and $V_{\gamma}$ has a parameter-freely power-OTM-computable code (which exists by definition of $\sigma^{\mathcal{P}}$). Now $\text{card}(\text{tc}(x))\leq\text{card}(V_{\gamma})$ and $\delta:=\text{card}(V_{\gamma})$ is parameter-freely power-OTM-computable using $P_{\text{card}}$. It follows that $\delta^{+}$ is also parameter-freely power-OTM-computable, and hence, so is (a code for) $V_{\delta^{+}}$; so $V_{\delta^{+}}\in V_{\sigma^{\mathcal{P}}}$. But now, by Lemma \ref{early witnesses}, if, for some $x\in X$, there exists $y$ such that $\psi(x,y)$, then such $y$ exists in $V_{\delta^{+}}$. Thus $V_{\delta^{+}}$ is as required by $\Delta_{0}$-collection.

%We now run through $X$. For each $x\in X$, we let $P_{\text{set-enum}}$ run until some $y$ with $\psi(x,y)$ is found. This $y$ is then stored on some scratch tape before we continue with the next element of $X$. Once the run through $X$ is complete, we have computed some set $Y$ such that $\forall{x\in X}\exists{y\in Y}\psi(x,y)$.\todo{Das funktioniert noch nicht so ganz, weil $X$ vielleicht nicht berechenbar ist. Das Ergebnis stimmt jedenfalls, Beweis noch überlegen.} 
%\todo{MORE HERE! CHECK THIS!} 
In the proof for (3), we gave an example for a function $f:\omega\rightarrow V_{\sigma^{\mathcal{P}}}$ that is $\Sigma_{2}$-definable over $V_{\sigma^{\mathcal{P}}}$ but has $f[\omega]\notin V_{\sigma^{\mathcal{P}}}$.
\end{enumerate}
\end{proof}

We can now show that our notion of ``computability with a power-set operator'' is different from the one proposed by Moss \cite{Moss}. Namely, in \cite{Moss}, power-recursive function are defined over so-called power-admissible sets, which are sets satisfying ZF with the replacement axiom restricted to $\Sigma_{1}$-formulas (see \cite{Moss}, p. 256).

\begin{corollary}
    Parameter-free power-OTM-computability is strictly stronger than power-recursion.
\end{corollary}
\begin{proof}
%\todo{Wenn das stimmt, sind Power-OTMs stärker als Power-recursion, weil die erste Power-admissible  unterhalb von $\sigma^{\mathcal{P}}$ liegen wird (denn ``Es gibt ein transitives Modell von Power-KP'' ist $\Sigma_{1}^{\mathcal{P}}$).}
The statement ``There is a transitive model of power-KP'' is clearly $\Sigma_{1}^{\mathcal{P}}$. Thus, the first power-admissible ordinal $\alpha$ is smaller than $\sigma^{\mathcal{P}}$.
\end{proof}

    In Moss \cite{Moss} p. 260, it is shown that, provided strongly inaccessible cardinals exist, the first ordinal $\alpha$ such that $V_{\alpha}$ is closed under power-recursive functions is smaller than the smallest such. The same is not true for power-OTMs: 
    
\begin{lemma}
Under the assumption that the respective cardinals exist, we have the following:
\begin{enumerate}
    \item For every $n\in\omega$, $\sigma^{\mathcal{P}}$ is greater than the $n$-th strongly inaccessible cardinal.
    \item $\sigma^{\mathcal{P}}$ is greater than the first strongly inaccessible limit of strongly inaccessible cardinals (i.e., the first $1$-inaccessible cardinal). 
    \item For every $k\in\omega$, $\sigma^{\mathcal{P}}$ is greater than the first $k$-inaccessible cardinal. 
\end{enumerate}
\end{lemma}
%Assuming that such cardinals do exist, it is easy to see \todo{DO IT!} that $\sigma^{P}$ is larger than the first, second etc. strongly inaccessible cardinal, larger than the first strongly inaccessible limit of such etc.\todo{evaluate Mahlo-ness}
\begin{proof}
\begin{enumerate}
    \item We enumerate $V$. For every ordinal $\alpha$ thus produced, we check (i) whether it is a regular cardinal (for this, it suffices to search through $\mathfrak{P}(\alpha)$) and (ii) whether $|\mathfrak{P}(\beta)|<\alpha$ for all $\beta<\alpha$ (which can be done by producing the power set of $\beta$ and then applying $P_{\text{card}}$ to it). By assumption, the answer will be positive at some point, at which we return $\alpha$, thus computing the first strongly inaccessible cardinal. 
    \item Using the routine described in the proof for (1), we can test a cardinal for being strongly inaccessible. We can then use this to search for a cardinal which is itself strongly inaccessible and has unboundedly many such cardinals below.
    \item Inductively, there is a power-OTM-program for deciding the class of $k$-inaccessible cardinals, for every $k\in\omega$ (in fact, using recursion there is such a program that does this uniformly in $k$). Then again, we can search for the first cardinal that has this property.
\end{enumerate}
\end{proof}

\begin{remark}
    The above could, of course, be further iterated.
\end{remark}

\section{Realizability with Power-OTMs}

Realizability is defined by replacing OTMs with arbitrary set parameters with power-OTMs; we obtain two different notions of realizability, depending on whether ordinal parameters are allowed or not; however, all of the below results hold true for both of these notions, for much the same reasons. We therefore do not emphasize the distinction very much. Many of the results are, again, analogous to those obtained by Passmann in his investigation of SRM-realizability (\cite{Passmann}, pp. 320--323). However, in the case of the axiom of choice, our setting differs from that of Passmann.

\begin{lemma} [Cf. Passmann, \cite{Passmann}]
The power-realizable statements include all axioms of KP,\footnote{See Passmann, \cite{Passmann}, Theorem 4.4.  The realizability of the KP axioms was already proved in \cite{CGP} for ordinary OTMs; the arguments carry over verbatim to power-OTMs.} and the power set axiom,\footnote{Passmann, \cite{Passmann}, Theorem 4.4.} and are closed under intuitionistic predicate logic, but not under classical predicate logic.\footnote{Cf. Passmann, \cite{Passmann}, Theorem 4.2.} 
There are instances of the separation axiom that are not power-realizable.\footnote{Cf. Passmann, \cite{Passmann}, Theorem 4.5.}
\end{lemma}

The last result makes it natural to ask how much of separation is power-realizable. This is answered by the next theorem.

As a preparation for the proof of part (2) of Theorem \ref{power realizable sep finetuning}, we need to define the halting problem for power-OTMs. Note that the definition is not obvious, as the definition of power-OTM-computations implicitly quantifies over codings. There are thus two ways to interpret the halting problem: Once as asking whether $P^{\hat{F}}(\rho)$ will halt for \textit{all} encodings $\hat{F}$ of $\mathfrak{P}$, and once as asking whether there is an encoding $\hat{F}$ of $\mathfrak{P}$ such that $P^{\hat{F}}(\rho)$ halts. For our purposes, it is the second variant that we need. 

\begin{defini}
The halting problem $\mathcal{H}_{\text{power}}$ for Power-OTMs with ordinal parameters is the class of all pairs $(P,\rho)$ where $P$ is a power-OTM-program, $\rho$ is an ordinal such that, for some coding $\hat{F}$ of $\mathfrak{P}$, we have $P^{\hat{F}}(\rho)$. 
\end{defini}

As this definition differs slightly from the standard way of defining a halting set, we make sure that the usual argument for the unsolvability of the halting problem still applies:

\begin{lemma}{\label{power halting unsolvable}}
There is no pair $(P,\rho)$ of a power-OTM-program $P$ and an ordinal $\rho$ so that, for every coding $\hat{F}$ of $\mathfrak{P}$, $P^{\hat{F}}(\rho)$ decides $\mathcal{H}_{\text{power}}$. 
\end{lemma}
\begin{proof}
Assume for a contradiction that $(P,\rho)$ is such a pair. 
Using $(P,\rho)$, we can then construct a power-OTM-program $D$ which, in the parameter $\rho$, shows the following behaviour for every input $(Q,\xi)$ ($Q$ a power-OTM-program, $\xi$ an ordinal) and every coding $\hat{F}$ of $\mathfrak{P}$: 
\begin{itemize}
\item If there is a coding $\hat{G}$ of $\mathfrak{P}$ such that $Q^{\hat{G}}(\xi)\downarrow$, then $D^{\hat{F}}((Q,\xi),\rho)\uparrow$.
\item If there is no coding $\hat{G}$ of $\mathfrak{P}$ such that $Q^{\hat{G}}(\xi)\downarrow$, then $D^{\hat{F}}((Q,\xi),\rho)\downarrow$.
\end{itemize}
Let $\hat{F}$ be a coding of $\mathfrak{P}$. We consider the behaviour of $D^{\hat{F}}((D,\rho),\rho)$.
\begin{itemize}
\item If $D^{\hat{F}}((D,\rho),\rho)\downarrow$, then there clearly exists a coding $\hat{G}$ of $\mathfrak{P}$ such that $P^{\hat{F}}((D,\rho),\rho)\downarrow$ (namely, $\hat{F}$); by definition, this means that $D^{\hat{F}}((D,\rho),\rho)\uparrow$, a contradiction. 
\item If $D^{\hat{F}}((D,\rho),\rho)\uparrow$ then, by definition of $D$, there must be a coding $\hat{G}$ of $\mathfrak{P}$ such that $D^{\hat{G}}((D,\rho),\rho)\downarrow$ (for it there was no such $\hat{G}$, then, by definition, $D^{\hat{F}}((D,\rho),\rho)$ would halt). But as in (1), this implies that $D^{\hat{G}}((D,\rho),\rho)\uparrow$, a contradiction.
\end{itemize}
\end{proof}

\begin{lemma}{\label{power halting sigma2}}
There is a $\Sigma_{2}$-Formel $\phi(u,v):\Leftrightarrow\exists{x}\forall{y}\psi(x,y,u,v)$ (where $\psi$ is $\Delta_{0}$) such that, for every power-OTM-program $P$ and every ordinal $\rho$, the formula $\phi(P,\rho)$ is true if and only if $(P,\rho)\in\mathcal{H}_{\text{power}}$. 
\end{lemma}
\begin{proof}

%Es sei $\phi(k,\rho)$ die Formel $\mathcal{P}_{k}(\rho)\downarrow$, die also besagt, dass das $k$-te Power-OTM-Programm (ggf. im Parameter $\rho$) hält. 
%Dann ist $\phi(k,\rho)$ eine $\Sigma_{2}$-Formel, denn es lässt sich ausdrücken als: %: Dazu bezeichnen wir für eine Zweiband-OTM-Konfiguration $s$ mit $t(s)$ die Menge, die in $s$ durch den Inhalt des zweiten Arbeitsbandes codiert wird. Weiter verstehen wir 
%Unter einer ``partiellen Codierung von $\mathfrak{P}$'' verstehen wir eine Menge $f$, die Teilmenge einer Codierung von $\mathfrak{P}$ ist. 
$\phi(P,\rho)$ can be expressed as a $\Sigma_{2}$-formula as follows:

There exist a set $f$ and a sequence $s=:(s_{\iota}:\iota<\alpha+1)$ of configurations (where $s_{\iota}=(t_{\iota}^{0},h_{\iota}^{0},...,t_{\iota}^{4},h_{\iota}^{4}, z_{\iota})$; here, $t_{\iota}^{i}$ and $h_{\iota}^{i}$ denote the tape content of and the head position on the $i$-th tape, for $1\leq i\leq 4$, and $z_{\iota}$ denotes the inner state at time $\iota$), such that, for every set $x$, the following are true: 
\begin{enumerate}
\item $f$ is a partial function from $\mathfrak{P}(\text{On})$ to $\mathfrak{P}(\text{On})$.
\item For every element $a\in\text{dom}(f)$, if $a^{\prime}$ is the set coded by $a$ and $b$ is the set coded by $f(a)$, then $x\subseteq a^{\prime}\Leftrightarrow x\in b$. %die von $a$ codierte Menge, und ist ferner $b$ die von $f(a)$ codierte Menge, so gilt $x\subseteq a^{\prime}$ genau dann, wenn $x\in b$.
\item $s_{0}$ is the starting configuration of $\mathcal{P}_{k}(\rho)$.
\item $s_{\alpha}$ is a halting configuration of $\mathcal{P}_{k}(\rho)$.
\item $s_{\lambda}$ is, for every limit ordinal $\lambda\leq\alpha$, , obtained through the liminf-rule from $(s_{\iota}:\iota<\lambda)$.
\item If the command to be carried out at time $\iota$ is not the power set command, then $s_{\iota+1}$ is obtained from $s_{\iota}$ by carrying out the respective Turing command. 
\item If the command to be carried out at time $\iota$ is the power set command and the state is to be changed to $z^{\prime}$, then $s_{\iota+1}=(t_{\iota}^{0},h_{\iota}^{0},f(t_{\iota}^{1},h_{\iota}^{1},z_{\iota},z)$.
%$s_{\iota+1}$ gehört zum Definitionsbereich von $f$ und $s_{\iota+1}=f(s_{\iota})$, falls der auszuführende 
%\item $x\in t(c_{\iota+1})\leftrightarrow x\subseteq t(s_{\iota})$ für alle $\iota\leq\alpha$, für die der auszuführende Befehl im Zustand $s_{\iota}$ der Potenzmengenoperator ist.
\end{enumerate}
With respect to (2), note that ``$a$ codes $a^{\prime}$'' can be expressed by stating that, for every transitive KP-model $M\ni a$, the formula that $a$ codes $a^{\prime}$ holds true in $M$, so that (2) is expressable as a $\Pi_{1}$-formula. For the other clauses, this is obvious. 

 Thus, $\phi(k,\rho)$ is expressable as a $\Sigma_{2}$-formula. 
\end{proof}

\begin{lemma}{\label{power realizable sep finetuning}}
    \begin{enumerate}
        %\item All axioms of KP are power-OTM-realizable.\footnote{See Passmann, \cite{Passmann}, Theorem 4.4.}
        %\item The power set axiom is power-realizable.\footnote{Passmann, \cite{Passmann}, Theorem 4.4.}
        \item All instances of the $\Sigma_{1}$-separation scheme are power-realizable.
        \item There are instances of the $\Sigma_{2}$-separation scheme which are not power-realizable.
        %\item The set of power-OTM-realizable statements is closed under intuitionistic predicate logic, but not under classical predicate logic.\footnote{Cf. Passmann, \cite{Passmann}, Theorem 4.2.}
    \end{enumerate}
\end{lemma}
\begin{proof}
\begin{enumerate}
%   \item The realizability of the KP axioms was already proved in \cite{CGP} for ordinary OTMs; the arguments carry over verbatim to power-OTMs.
%   \item Trivial. 
   \item Let a $\Sigma_{1}$-formula $\phi(y):\Leftrightarrow\exists{x}\psi(x,y)$ be given (where $\psi$ only contains bounded quantifiers), along with a set $X$. We want to compute (a code for) the set $\{y\in X:\phi(y)\}$.

The algorithm now works as follows: Let a set $y$ be given. We first compute the transitive hull $t:=\text{tc}(\{y\})$ (which is possible on a parameter-free standard OTM, i.e., without the power set operator), and, using Lemma \ref{enumerating the universe}(3), $\kappa:=\text{card}(t)$. By iterating the power set operator along $\kappa$, we compute a code for $V_{\kappa}$, and from this a code for $V_{\kappa^{+}}$. By Lemma \ref{early witnesses}, $V_{\kappa^{+}}$ contains some $x$ with $\psi(x,y)$, if one exists. Thus, it suffices to search $V_{\kappa^{+}}$ for such an $x$ in order to determine the truth value of $\exists{x}\psi(x,y)$. If so, $y$ becomes part of our set, otherwise, it does not.

   \item As Passmann does it in his proof of \cite{Passmann}, Theorem 4.5 (and in analogy with what was done for ``plain'' OTMs in \cite{CarlBuch}, Proposition 9.4.4), we can use the formula expressing the $k$-th power-OTM-program will halt in the parameter $\rho\in\text{On}$. It was checked in Lemma \ref{power halting sigma2} that this is $\Sigma_{2}$. %It only needs to be checked that this formula is in fact $\Sigma_{2}$:

Now, if $P$ was a power-OTM-program and $\rho\in\text{On}$ a parameter such that $(P,\rho)\Vdash_{\text{OTM}}^{\textbf{P}}\forall{X\subseteq\omega\times\text{On}}\exists{Y}\forall{z}(z\in X\leftrightarrow(z\in X\wedge\phi(z)))$, then, for every given pair $(k,\xi)\in\omega\times\text{On}$, the question whether $\mathcal{P}_{k}(\xi)\downarrow$ would be decidable on a power-OTM by computing the set separated from $\{(k,\xi)\}$ with $\phi$ and checking whether it is empty. But, as we saw in Lemma \ref{power halting unsolvable}, this question is not decidable on a power-OTM.
%But the usual diagonalization argument shows that this question cannot be power-OTM-decidable. 
   
%   \item The proof for OTM-realizability in \cite{CGP} carries over without major modifications. Taking $\phi(P,\rho)$ to be the statement ``The power-OTM-program $P$, when run in the parameter $\rho$, will eventually halt or not halt'', it is clear that $\phi(p,\rho)$ is classically a logical tautology, but not power-OTM-realizable (since a realizer would amount to a halting problem solver for power-OTMs on a power-OTM). 
\end{enumerate}
\end{proof}

\begin{remark}
    The same results hold verbatim for power-$r$-realizability, which, as we have seen, is identical to power-OTM-realizability.
\end{remark}

In Passmann's setting, the realizability of the axiom of choice is trivial, as SRMs have access to a choice operator as a primitive operation (\cite{Passmann}, p. 312). It is noted in Theorem 4.7 that all $\Pi_{2}$-consequences of CZF have an SRM$^{+}$-computable witness function. We note here that the assumption can be weakened:

\begin{proposition}
Every formula of the form $\forall{x}\exists{y}\phi(x,y)$ (with $\phi$ a $\Delta_{0}$-formula) which is true in $V$ is SRM$^{+}$-realizable. In fact, the same holds for all formulas of the form $\exists{x}\forall{y}\exists{z}\phi(x,y,z)$ (with $\phi$ a $\Delta_{0}$-formula).\footnote{This result is the analogue of \cite{CGP}, Theorem $48$, where the same was shown for plain OTMs restricted to $L$.}
\end{proposition}
\begin{proof}
Suppose that $\forall{x}\exists{y}\phi(x,y)$ is true in $V$. Given $x$, enumerate the universe as explained in Lemma \ref{enumerating the universe} and search each new $V$-level for a $y$ such that $\phi(x,y)$ (using the ability of SRMs to evaluate bounded truth predicates). Once a $V$-level containing such $y$ has been found, use the ``TAKE'' operation to pick the minimal such $y$. 

Now suppose that $\exists{x}\forall{y}\exists{z}\phi(x,y,z)$ is true in $V$. Pick an appropriate $x$. This can be given to the program in advance. All that remains is then to realize $\forall{y}\exists{z}\phi(x,y,z)$, which is a true statement that can now be realized by the first part.
\end{proof}

In most formulations, the realizability of the axiom of choice is trivial (see \cite{CGP}, Theorem $50$). A non-trivial formulation is the one given in \cite{CGP}, p. 24: 
$$\forall{X}(\forall{x\in X}\neg\forall{y}y\notin x\rightarrow \exists{f:X\rightarrow\bigcup X}\forall{x\in X}f(x)\in x).$$ 
We will use AC to denote this formulation below. Even in this formulation, AC is SRM$^{+}$-realizable due to the build-in picking operation. This, however, is not the case for power-OTMs:

\begin{lemma}{\label{ac not power realizable}}
The power-realizability of AC is independent of ZFC.
\end{lemma}
\begin{proof}
It is easy to see that $\forall{x\in X}\neg\forall{y}y\notin x$ is power-realizable if and only if it is OTM-realizable if and only if it is true, in which case it is, e.g., power-realizable by the program returning $0$ on every input. We thus only need to consider the question whether choice functions for families of non-empty sets are power-OTM-computable. This was shown not to be the case if $V\neq\text{HOD}$ in \cite{Ca2024A}, Theorem 20(3). On the other hand, it is true if $V=L$.
\end{proof}

\section{Conclusion and further work}

We have seen that regarding the power set operation as a ``basic operation performed on sets'' (Rathjen \cite{Rathjen:Power}, p. 3) does indeed strongly increase the computational power of OTMs, but does not considerably alter their ability to realize impredicative instances of the separation scheme. It would be natural to consider notions of effective reducibility such as those considered in \cite{Ca2016}, \cite{Ca2018} for OTMs based on power-OTMs. It would also be interesting to study the structure of the parameter-freely power-OTM-clockable ordinals. Finally, one could, in analogy with the work of Moss \cite{Moss}, study the structure of the parameter-free power-OTM-degrees.

\end{document}